\documentclass{amsart}
\usepackage{amssymb,amsmath}
\theoremstyle{plain}
\newtheorem{thm}{Theorem}[section]

\newtheorem{cor}[thm]{Corollary}
\newtheorem{lem}[thm]{Lemma}
\newtheorem{prop}[thm]{Proposition}

\theoremstyle{definition}
\newtheorem{defn}{Definition}[section]

\theoremstyle{remark}
\newtheorem{rem}{Remark}[section]

\numberwithin{equation}{section}

\newcommand{\ra}{\rightarrow}
\newcommand{\Ra}{\Rightarrow}

\newcommand{\Lra}{\Leftrightarrow}

\begin{document}

\title[Comparison of cutoffs for Markov chains]{Comparison of cutoffs between lazy walks and Markovian semigroups}

\author[G.-Y. Chen]{Guan-Yu Chen$^1$}

\author[L. Saloff-Coste]{Laurent Saloff-Coste$^2$}

\thanks{$^1$Partially supported by NSC grant NSC100-2918-I-009-010}

\address{$^1$Department of Applied Mathematics, National Chiao Tung University, Hsinchu 300, Taiwan}
\email{gychen@math.nctu.edu.tw}

\thanks{$^2$Partially supported by NSF grant DMS-1004771}

\address{$^2$Malott Hall, Department of Mathematics, Cornell University, Ithaca, NY 14853-4201}
\email{lsc@math.cornell.edu}



\begin{abstract}
We make a connection between the continuous time and lazy discrete time Markov chains through the comparison of cutoffs and mixing time in total variation distance. For illustration, we consider finite birth and death chains and provide a criterion on cutoffs using eigenvalues of the transition matrix.
\end{abstract}

\maketitle

\section{Introduction}\label{s-intro}

Let $\Omega$ be a countable set and $(\Omega,K,\pi)$ be an irreducible Markov chain on $\Omega$ with transition matrix $K$ and stationary distribution $\pi$. Let
\[
 H_t=e^{-t(I-K)}=\sum_{i=0}^\infty e^{-t}t^iK^i/i!
\]
be the associated semigroup which describes the corresponding natural continuous time process on $\Omega$. For $\delta\in(0,1)$, set
\begin{equation}\label{eq-lazyK}
    K_\delta=\delta I+(1-\delta) K,
\end{equation}
where $I$ is the identity matrix indexed by $\Omega$. Clearly, $K_\delta$ is similar to $K$ but with an additional holding probability depending of $\delta$. We call $K_\delta$ the $\delta$-lazy walk or $\delta$-lazy chain of $K$. It is well-known that if $K$ is irreducible with stationary distribution $\pi$, then
\[
    \lim_{m\ra\infty}K_\delta^m(x,y)=\lim_{t\ra\infty}H_t(x,y)=\pi(y),\quad\forall  x,y\in\Omega,\,\delta\in(0,1).
\]

In this paper, we consider convergence in total variation. The total variation between two probabilities $\mu,\nu$ on $\Omega$ is defined by $\|\mu-\nu\|_{\text{\tiny TV}}=\sup\{\mu(A)-\nu(A)|A\subset\Omega\}$. For any irreducible $K$ with stationary distribution $\pi$, the (maximum) total variation distance is defined by
\begin{equation}\label{eq-tv}
    d_{\text{\tiny TV}}(m)=\sup_{x\in\Omega}\|K^m(x,\cdot)-\pi\|_{\text{\tiny TV}},
\end{equation}
and the corresponding mixing time is given by
\begin{equation}\label{eq-tvmix}
    T_{\text{\tiny TV}}(\epsilon)=\inf\{m\ge 0|d_{\text{\tiny TV}}(m)\le\epsilon\}.
\end{equation}
We write the total variation distance and mixing time as $d_{\text{\tiny TV}}^{(c)},T_{\text{\tiny TV}}^{(c)}$ for the continuous semigroup and as $d_{\text{\tiny TV}}^{(\delta)},T_{\text{\tiny TV}}^{(\delta)}$ for the $\delta$-lazy walk.

A sharp transition phenomenon, known as cutoff, was introduced by Aldous and Diaconis in early 1980s. See e.g. \cite{D96cutoff,CSal08} for an introduction and a general review of cutoffs. In total variation, a family of irreducible Markov chains $(\Omega_n,K_n,\pi_n)_{n=1}^\infty$ is said to present a cutoff if
\begin{equation}\label{eq-cutmix}
    \lim_{n\ra\infty}\frac{T_{n,\text{\tiny TV}}(\epsilon)}{T_{n,\text{\tiny TV}}(\eta)}=1,\quad\forall 0<\epsilon<\eta<1.
\end{equation}
The family is said to present a $(t_n,b_n)$ cutoff if $b_n=o(t_n)$ and
\[
 |T_{n,\text{\tiny TV}}(\epsilon)-t_n|=O(b_n),\quad\forall 0<\epsilon<1.
\]
The cutoff for the associated continuous semigroups is defined in a similar way.
This paper contains the following general result.
\begin{thm}\label{t-main}
Consider a family of irreducible and positive recurrent Markov chains $\mathcal{F}=\{(\Omega_n,K_n,\pi_n)|n=1,2,...\}$. For $\delta\in(0,1)$, let $\mathcal{F}_\delta$ be the family of associated $\delta$-lazy walks and let $\mathcal{F}_c$ be the family of associated continuous semigroups. Suppose $T_{n,\textnormal{\tiny TV}}^{(c)}(\epsilon_0)\ra\infty$ for some $\epsilon_0\in(0,1)$. Then, the following are equivalent.
\begin{itemize}
\item[(1)] $\mathcal{F}_\delta$ has a cutoff in total variation.

\item[(2)] $\mathcal{F}_c$ has a cutoff in total variation.
\end{itemize}
Furthermore, if $\mathcal{F}_c$ has a cutoff, then
\[
    \lim_{n\ra\infty}\frac{T_{n,\textnormal{\tiny TV}}^{(c)}(\epsilon)}{T_{n,\textnormal{\tiny TV}}^{(\delta)}(\epsilon)}=1-\delta,\quad\forall\epsilon\in(0,1).
\]
\end{thm}

\begin{thm}\label{t-main2}
Let $\mathcal{F}$ be the family in Theorem \ref{t-main}. Assume that $t_n\ra\infty$. Then, the following are equivalent.
\begin{itemize}
\item[(1)] $\mathcal{F}_c$ has a $(t_n,b_n)$ cutoff.

\item[(2)] For $\delta\in(0,1)$, $\mathcal{F}_\delta$ has a $(t_n/(1-\delta),b_n)$ cutoff.
\end{itemize}
\end{thm}
We refer the readers to Theorems \ref{t-cutcomp}, \ref{t-window}, \ref{t-wcomp} and \ref{t-ini} for more detailed discussions.

For an illustration, we consider finite birth and death chains. For $n\ge 1$, let $\Omega_n=\{0,1,...,n\}$ and $K_n$ be the transition kernel of a birth and death chain on $\Omega_n$ with birth rate $p_{n,i}$, death rate $q_{n,i}$ and holding rate $r_{n,i}$, where $p_{n,n}=q_{n,0}=0$ and $p_{n,i}+q_{n,i}+r_{n,i}=1$. Suppose that $K_n$ is irreducible with stationary distribution $\pi_n$. For the family $\{(\Omega_n,K_n,\pi_n)|n=1,2,...\}$, Ding {\it et al.} \cite{DLP10} showed that, in the discrete time case, if $\inf_{i,n}r_{n,i}>0$, then the cutoff in total variation exists if and only if the product of the total variation mixing time and the spectral gap, which is defined to be the smallest non-zero eigenvalue of $I-K$, tends to infinity. There is also a similar version for the continuous time case. The next theorem is an application of the above result and Theorem \ref{t-main}, which is summarized from Theorem \ref{t-mixingtvsep3}.

\begin{thm}\label{t-bdc-cut-main}
Let $\mathcal{F}=\{(\Omega_n,K_n,\pi_n)|n=1,2,...\}$ be a family of irreducible birth and death chains as above. For $n\ge 1$, let $0,\lambda_{n,1},...,\lambda_{n,n}$ be eigenvalues of $I-K_n$ and set
\[
 \lambda_n=\min_{1\le i\le n}\lambda_{n,i},\quad s_n=\sum_{i=1}^n\lambda_{n,i}^{-1}.
\]
Then, the following are equivalent.
\begin{itemize}
\item[(1)] $\mathcal{F}_c$ has a total variation cutoff.

\item[(2)] For $\delta\in(0,1)$, $\mathcal{F}_\delta$ has a total variation cutoff.

\item[(3)] $s_n\lambda_n\ra\infty$.
\end{itemize}
\end{thm}

The remaining of this article is organized as follows. In Section 2, the concepts of cutoffs and mixing times are introduced and fundamental results are reviewed. In Section 3, a detailed comparison of the cutoff time and window size is made between the continuous time and lazy discrete time cases, where the state space is allowed to be infinite. In Section 4, we focus on finite birth and death chains and provide a criterion on total variation cutoffs using the eigenvalues of the transition matrices.

\section{Cutoffs in total variation}

Throughout this paper, for any two sequences $s_n,t_n$ of positive numbers, we write $s_n=O(t_n)$ if there are $C>0,N>0$ such that $|s_n|\le C|t_n|$ for $n\ge N$. If $s_n=O(t_n)$ and $t_n=O(s_n)$, we write $s_n\asymp t_n$. If $t_n/s_n\ra 1$ as $n\ra\infty$, we write $t_n\sim s_n$.

Consider the following definitions.
\begin{defn}\label{d-cutoff}
Referring to the notation in (\ref{eq-tv}), a family $\mathcal{F}=\{(\Omega_n,K_n,\pi_n)|n=1,2,...\}$ is said to present a total variation
\begin{itemize}
\item[(1)] precutoff if there is a sequence $t_n$ and $B>A>0$ such that
\[
    \lim_{n\ra\infty}d_{n,\text{\tiny TV}}(\lceil Bt_n\rceil)=0,\quad\liminf_{n\ra\infty}d_{n,\text{\tiny TV}}(\lfloor At_n\rfloor)>0.
\]

\item[(2)] cutoff if there is a sequence $t_n$ such that, for all $\epsilon>0$,
\[
    \lim_{n\ra\infty}d_{n,\text{\tiny TV}}(\lceil(1+\epsilon)t_n\rceil)=0,\quad\lim_{n\ra\infty}d_{n,\text{\tiny TV}}(\lfloor(1-\epsilon)t_n\rfloor)=1.
\]

\item[(3)] $(t_n,b_n)$ cutoff if $b_n=o(t_n)$ and
\[
    \lim_{c\ra\infty}\overline{F}(c)=0,\quad\lim_{c\ra-\infty}\underline{F}(c)=1,
\]
where
\[
    \overline{F}(c)=\limsup_{n\ra\infty}d_{n,\text{\tiny TV}}(\lceil t_n+cb_n\rceil),\quad
    \underline{F}(c)=\liminf_{n\ra\infty}d_{n,\text{\tiny TV}}(\lfloor t_n+cb_n\rfloor).
\]
\end{itemize}
\end{defn}
In definition \ref{d-cutoff}, $t_n$ is called a cutoff time and $b_n$ is called a window for $t_n$. The cutoffs for continuous semigroups is the same except the deletion of $\lceil\cdot\rceil$ and $\lfloor\cdot\rfloor$.

\begin{rem}
In Definition \ref{d-cutoff}, if $t_n\ra\infty$ (or equivalently $T_{n,\text{\tiny TV}}(\epsilon)\ra\infty$ for some $\epsilon\in(0,1)$), then the cutoff is consistent with (\ref{eq-cutmix}). This is also true for cutoffs in continuous semigroups without the assumption $t_n\ra\infty$.
\end{rem}

The following lemma characterizes the total variation convergence using specific subsequences of indices and events, which is useful in proving and disproving cutoffs.

\begin{lem}\label{l-cut1}
Consider a family of irreducible and positive recurrent Markov chains $\{(\Omega_n,K_n,\pi_n)|n=1,2,...\}$. Let $t_n$ be a sequence of nonnegative integers. Then, the following are equivalent.
\begin{itemize}
\item[(1)] $d_{n,\textnormal{\tiny TV}}(t_n)\ra 0$.

\item[(2)] For any increasing sequence of positive integers $n_k$, any $A_{n_k}\subset\Omega_{n_k}$ and any $x_{n_k}\in\Omega_{n_k}$, there is a subsequence $m_k$ such that
\[
    \lim_{k\ra\infty}\left|K_{m_k}^{t_{m_k}}(x_{m_k},A_{m_k})-\pi_{m_k}(A_{m_k})\right|=0.
\]
\end{itemize}
\end{lem}

\begin{proof}[Proof of Lemma \ref{l-cut1}]
(1)$\Ra$(2) is obvious. For (2)$\Ra$(1), choose $A_n\subset\Omega_n$ and $x_n\in\Omega_n$ such that $d_{n,\text{\tiny TV}}(t_n)\le 2|K_n^{t_n}(x_n,A_n)-\pi_n(A_n)|$. Let $n_k$ be an increasing sequence of positive integers and choose a subsequence $m_k$ such that
\[
    \lim_{k\ra\infty}\left|K_{m_k}^{t_{m_k}}(x_{m_k},A_{m_k})-\pi_{m_k}(A_{m_k})\right|=0.
\]
This implies $d_{m_k,\text{\tiny TV}}(t_{m_k})\ra 0$, as desired.
\end{proof}

\begin{rem}\label{r-cut}
Lemma \ref{l-cut1} also holds in continuous time under the release of $t_n$ to positive real numbers. See \cite{GY-PhD,CSal08} for further discussions on cutoffs.
\end{rem}

\section{Comparisons of cutoffs}\label{s-cutoff}

In this section, we establish the relation of cutoffs between lazy walks and continuous semigroups. Let $\Omega$ be a countable set and $K$ be a transition matrix indexed by $\Omega$. In the notation of (\ref{eq-lazyK}), the $\delta$-lazy walk evolves in accordance with
\[
    (K_\delta)^t=\sum_{i=0}^t\binom{t}{i}\delta^{t-i}(1-\delta)^iK^i,\quad\forall \delta\in(0,1),\,t\ge 0,
\]
whereas the continuous time chain follows
\[
    H_t=e^{-t(I-K)}=\sum_{i=0}^\infty\left(e^{-t}\frac{t^i}{i!}\right)K^i.
\]
Observe that $I-K=(I-K_\delta)/(1-\delta)$. This implies
\begin{equation}\label{eq-upper}
    d_{\text{\tiny TV}}^{(c)}(t)\le e^{-t/(1-\delta)}\sum_{i=0}^m\frac{[t/(1-\delta)]^i}{i!}+d_{\text{\tiny TV}}^{(\delta)}(m).
\end{equation}
Concerning the cutoff times and windows, we discuss each of them in detail.

\subsection{Cutoff times}\label{ss-cutofftime}

\begin{thm}\label{t-cutcomp}
Let $\mathcal{F}=\{(\Omega_n,K_n,\pi_n)|n=1,2,...\}$ be a family of irreducible Markov chains on countable state spaces with stationary distributions. For $\delta\in(0,1)$, let $\mathcal{F}_\delta=\{(\Omega_n,K_{n,\delta},\pi_n)|n=1,2,...\}$ and $\mathcal{F}_c=\{(\Omega_n,H_{n,t},\pi_n)|n=1,2,...\}$. Suppose there is $\epsilon_0>0$ such that $T_{n,\textnormal{\tiny TV}}^{(\delta)}(\epsilon_0)\ra\infty$ or $T_{n,\textnormal{\tiny TV}}^{(c)}(\epsilon_0)\ra\infty$. Then, the following are equivalent.
\begin{itemize}
\item[(1)] $\mathcal{F}_\delta$ has a cutoff (resp. precutoff) in total variation.

\item[(2)] $\mathcal{F}_c$ has a cutoff (resp. precutoff) in total variation.
\end{itemize}
Furthermore, if $\mathcal{F}_c$ has a cutoff, then
\[
    \lim_{n\ra\infty}\frac{T_{n,\textnormal{\tiny TV}}^{(c)}(\epsilon)}{T_{n,\textnormal{\tiny TV}}^{(\delta)}(\epsilon)}=1-\delta,\quad\forall\epsilon\in(0,1).
\]
\end{thm}

The above theorem is in fact a simple corollary of the following proposition.

\begin{prop}\label{p-main}
Let $\mathcal{F}_\delta,\mathcal{F}_c$ be families in Theorem \ref{t-cutcomp} and $t_n,r_n$ be sequences tending to infinity. Fix $\delta\in(0,1)$.
\begin{itemize}
\item[(1)] If $d_{n,\textnormal{\tiny TV}}^{(\delta)}(\lceil t_n\rceil)\ra 0$, then
\[
    \lim_{n\ra\infty}d_{n,\textnormal{\tiny TV}}^{(c)}((1-\delta)t_n+cb_n)=0,
\]
for all $c>0$ and for any sequence $b_n$ satisfying $\sqrt{t_n}=o(b_n)$.

\item[(2)] If $d_{n,\textnormal{\tiny TV}}^{(c)}(r_n)\ra 0$, then
\[
    \lim_{n\ra\infty}d_{n,\textnormal{\tiny TV}}^{(\delta)}(\lceil r_n/(1-\delta)+cb_n\rceil)=0,
\]
for all $c>0$ and for any sequence $b_n$ satisfying $\sqrt{r_n}=o(b_n)$.

\item[(3)] If $d_{n,\textnormal{\tiny TV}}^{(c)}(r_n)\ra 1$, then
\[
    \lim_{n\ra\infty}d_{n,\textnormal{\tiny TV}}^{(\delta)}(\lfloor r_n/(1-\delta)\rfloor)=1.
\]

\item[(4)] If $d_{n,\textnormal{\tiny TV}}^{(\delta)}(\lfloor t_n\rfloor)\ra 1$, then
\[
    \lim_{n\ra\infty}d_{n,\textnormal{\tiny TV}}^{(c)}((1-\delta)t_n)=1.
\]
\end{itemize}
\end{prop}
\begin{proof}
We prove (1), while (2) goes in a similar way and is omitted. Suppose $d_{n,\text{\tiny TV}}^{(\delta)}(\lceil t_n\rceil)\ra 0$. Since $\sqrt{t_n}=o(b_n)$, it is clear that
\begin{equation}\label{eq-cc'}
    \lim_{n\ra\infty}d_{n,\text{\tiny TV}}^{(\delta)}(\lceil t_n+cb_n+c'\sqrt{t_n}\rceil)=0,\quad\forall c>0,\,c'\in\mathbb{R}.
\end{equation}
Fix $c>0$ and let $x_n\in\Omega_n,A_n\subset\Omega_n$. Given any increasing sequence $n_l$, we may choose, according to Lemma \ref{l-meas}, a subsequence $m_l$ such that $\pi_{m_l}(A_{m_l})\ra\alpha\in[0,1]$ and, for all $c'\in\mathbb{R}$,
\[
    \lim_{l\ra\infty}K_{m_l,\delta}^{\lceil t_{m_l}+cb_{m_l}+c'\sqrt{t_{m_l}}\rceil}(x_{m_l},A_{m_l})=\frac{1}{\sqrt{2\pi\delta}}
    \int_{-\infty}^\infty e^{-(x-c')^2/(2\delta)}f(x)dx,
\]
and
\[
    \lim_{l\ra\infty}H_{m_l,(1-\delta)(t_{m_l}+cb_{m_l})}(x_{m_l},A_{m_l})=\frac{1}{\sqrt{2\pi}}
    \int_{-\infty}^\infty e^{-x^2/2}f(x)dx,
\]
where $f$ is nonnegative and bounded by $1$. By (\ref{eq-cc'}) and Lemma \ref{l-conv}, $f$ equals to $\alpha$ almost everywhere and, by Lemma \ref{l-cut1}, this implies $d_{n,\text{\tiny TV}}^{(c)}((1-\delta)t_n+cb_n)\ra 0$ as $n\ra\infty$ for all $c>0$.

The proofs for (3) and (4) are similar and we only give the details for (4). First, we choose sequences $x_n\in\Omega_n$ and $A_n\subset\Omega_n$ such that
\[
    \lim_{n\ra\infty}\pi_n(A_n)=1,\quad \lim_{n\ra\infty}K_{n,\delta}^{\lfloor t_n\rfloor}(x_n,A_n)=0.
\]
Let $n_l$ be a sequence tending to infinity. Applying Lemma \ref{l-meas} with $c=0$ and $a_{n,m}=K_n^m(x_n,A_n)$, we may choose a subsequence, say $m_l$, such that
\[
    \lim_{l\ra\infty}H_{m_l,(1-\delta)t_{m_l}}(x_{m_l},A_{m_l})
    =\frac{1}{\sqrt{2\pi}}\int_{-\infty}^\infty e^{-x^2/2}g(x)dx
\]
and
\[
    \lim_{l\ra\infty}K_{m_l,\delta}^{\lfloor t_{m_l}\rfloor}(x_{m_l},A_{m_l})
    =\frac{1}{\sqrt{2\pi\delta}}\int_{-\infty}^\infty e^{-x^2/(2\delta)}g(x)dx,
\]
where $g$ is nonnegative measurable function bounded by $1$. This leads to $g=0$ almost everywhere and
\[
    \lim_{l\ra\infty}d_{m_l,\text{\tiny TV}}^{(c)}((1-\delta)t_{m_l})=1.
\]
\end{proof}

The following is a simple corollary of Proposition \ref{p-main} (1)-(2).
\begin{cor}\label{c-main}
Let $\mathcal{F}_\delta,\mathcal{F}_c$ be families in Theorem \ref{t-cutcomp} and $t_n,r_n$ be sequences tending to infinity. Fix $\delta\in(0,1)$.
\begin{itemize}
\item[(1)] If $d_{n,\textnormal{\tiny TV}}^{(\delta)}(\lceil t_n\rceil)\ra 0$, then
\[
    \lim_{n\ra\infty}d_{n,\textnormal{\tiny TV}}^{(c)}((1+\epsilon)(1-\delta)t_n)=0,\quad\forall \epsilon>0.
\]

\item[(2)] If $d_{n,\textnormal{\tiny TV}}^{(c)}(r_n)\ra 0$, then
\[
    \lim_{n\ra\infty}d_{n,\textnormal{\tiny TV}}^{(\delta)}(\lceil(1+\epsilon)r_n/(1-\delta)\rceil)=0,\quad\forall \epsilon>0.
\]
\end{itemize}
\end{cor}

\begin{proof}[Proof of Theorem \ref{t-cutcomp}]
Set $r_n=T_{n,\text{\tiny TV}}^{(\delta)}(\epsilon_0)$ and $s_n=T_{n,\text{\tiny TV}}^{(c)}(\epsilon_0)$. Suppose $r_n\ra\infty$. By Corollary \ref{c-main} (2), if
\[
    \liminf_{n\ra\infty}d_{n,\text{\tiny TV}}^{(c)}((1-\delta)r_n/2)=0,
\]
then
\[
    \liminf_{n\ra\infty}d_{n,\text{\tiny TV}}^{(\delta)}(\lceil(1+\epsilon)r_n/2\rceil)=0,\quad\forall \epsilon>0.
\]
But, taking $\epsilon=1/2$ implies that, for $n$ large enough,
\[
    d_{n,\text{\tiny TV}}^{(\delta)}(\lceil(1+\epsilon)r_n/2\rceil)\ge d_{n,\text{\tiny TV}}^{(\delta)}(r_n-1)>\epsilon_0>0.
\]
This makes a contradiction and, hence, if $r_n\ra\infty$, then
\[
    \liminf_{n\ra\infty}d_{n,\text{\tiny TV}}^{(c)}((1-\delta)r_n/2)>0.
\]
In a similar way, if $s_n\ra\infty$, then Corollary \ref{c-main} (1) implies
\[
    \liminf_{n\ra\infty}d_{n,\text{\tiny TV}}^{(\delta)}(\lceil s_n\rceil)>0.
\]
This proves the following equivalence.
\[
    T_{n,\text{\tiny TV}}^{(\delta)}(\epsilon_0)\ra\infty\quad\text{for some }\epsilon_0>0\quad\Lra\quad
    T_{n,\text{\tiny TV}}^{(c)}(\epsilon_0)\ra\infty\quad\text{for some }\epsilon_0>0.
\]

For the equivalence of (1) and (2), the proof for precutoffs is given by Corollary \ref{c-main} (1)-(2), while the proof for cutoffs also uses Proposition \ref{p-main} (3)-(4).
\end{proof}

\subsection{Cutoff windows}\label{ss-window}

This section is devoted to the comparison of cutoff windows introduced in Definition \ref{d-cutoff}.

\begin{thm}\label{t-window}
Let $\mathcal{F}$ be a family of irreducible positive recurrent Markov chains and $\mathcal{F}_\delta,\mathcal{F}_c$ be associated families of lazy walks and continuous semigroups. Let $t_n,b_n$ be sequences of positive reals and assume that $t_n\ra\infty$. If $\mathcal{F}_\delta$ (resp. $\mathcal{F}_c$) presents a $(t_n,b_n)$ cutoff in total variation, then $\sqrt{t_n}=O(b_n)$.
\end{thm}
\begin{rem}\label{r-window}
There are examples with cutoffs but the order of any window size must be bigger than $\sqrt{t_n}$. Consider the Ehrenfest chain on $\{0,...,n\}$, which is a birth and death chain with rates $p_{n,i}=1-i/n$, $q_{n,i}=i/n$ and $r_{n,i}=0$. It is obvious that $K_n$ is irreducible and periodic with stationary distribution $\pi_n(i)=2^{-n}\binom{n}{i}$. An application of the representation theory shows that, for $0\le i\le n$, $2i/n$ is an eigenvalue of $I-K_n$. Let $\lambda_n=2/n$ and $s_n=\sum_{i=1}^nn/(2i)=\frac{1}{2}n\log n+O(n)$. By Theorem \ref{t-bdc-cut}, since $\lambda_ns_n$ tends to infinity, both $\mathcal{F}_c$ and $\mathcal{F}_\delta$ have a total variation cutoff. For a detailed computation on the total variation and the $L^2$-distance, see e.g. \cite{D88}. It is well-known that $\mathcal{F}_c$ has a $(\frac{1}{4}n\log n,n)$ total variation cutoff. By Theorem \ref{t-wcomp}, $\mathcal{F}_\delta$ has a $(\frac{n\log n}{4(1-\delta)},n)$ total variation cutoff for $\delta\in(0,1)$, which is nontrivial. For the continuous time Ehrenfest chains, Theorem \ref{t-window} says that the window size is at least $\sqrt{n\log n}$, while $n$ is the correct order.
\end{rem}

\begin{proof}[Proof of Theorem \ref{t-window}]
We prove the continuous time case. The lazy discrete time case can be treated similarly. Assume the inverse that the sequence $\sqrt{t_n}/b_n$ is not bounded. By considering the subsequence of $\sqrt{t_n}/b_n$ which tends to infinity, it loses no generality to assume that $b_n=o(\sqrt{t_n})$. According to the definition of cutoffs, we may choose $C>0$, $x_n\in\Omega_n$ and $A_n\subset\Omega_n$ such that
\[
    \liminf_{n\ra\infty}|H_{n,t_n+Cb_n}(x_n,A_n)-\pi_n(A_n)|>0.
\]
By Lemma \ref{l-meas}, one may choose a sequence $n_l$ tending to infinity such that $\pi_{n_l}(A_{n_l})$ converges to $\alpha\in[0,1]$ and
\[
    \lim_{l\ra\infty}H_{n_l,t_{n_l}+Cb_{n_l}}(x_{n_l},A_{n_l})=\frac{1}{\sqrt{2\pi}}\int_{-\infty}^\infty
    e^{-x^2/2}f(x)dx\ne\alpha,
\]
where $f$ is positive and bounded by $1$. Let $c\in\mathbb{R}$. For any $\epsilon>0$, choose $N>0$ such that, for $n\ge N$,
\[
    \left|H_{n,t_n+cb_n}(x_n,A_n)-\sum_{i:|i-t_n|\le N\sqrt{t_n}}\left(e^{-(t_n+cb_n)}\frac{(t_n+cb_n)^i}{i!}\right)
    K_n^i(x_n,A_n)\right|<\epsilon.
\]
Note that
\[
    e^{-(t_n+cb_n)}\frac{(t_n+cb_n)^i}{i!}=e^{-(t_n+Cb_n)}\frac{(t_n+Cb_n)^i}{i!}(1+o(1))\quad\text{as }n\ra\infty,
\]
where $o(1)$ is uniform for $|i-t_n|\le N\sqrt{t_n}$. This implies
\[
    \lim_{l\ra\infty}H_{n_l,t_{n_l}+cb_{n_l}}(x_{n_l},A_{n_l})=\frac{1}{\sqrt{2\pi}}\int_{-\infty}^\infty e^{-x^2/2}f(x)dx,\quad\forall c\in\mathbb{R}.
\]
Since $\mathcal{F}_c$ presents a $(t_n,b_n)$ cutoff, the right-side integral is equal to $\alpha$, a contradiction.
\end{proof}

\begin{thm}\label{t-wcomp}
Let $\mathcal{F}_\delta,\mathcal{F}_c$ be families in Theorem \ref{t-window} and $t_n\ra\infty$. Then, the following are equivalent.
\begin{itemize}
\item[(1)] $\mathcal{F}_\delta$ has a $(t_n,b_n)$ cutoff.

\item[(2)] $\mathcal{F}_c$ has a $((1-\delta)t_n,b_n)$ cutoff.
\end{itemize}
\end{thm}

To prove this theorem, we need the following proposition.

\begin{prop}\label{p-wcomp}
Let $\mathcal{F}_\delta,\mathcal{F}_c$ be as in Theorem \ref{t-wcomp} and $t_n,r_n$ be sequences tending to infinity.
\begin{itemize}
\item[(1)] If $\mathcal{F}_\delta$ has a $(t_n,b_n)$ cutoff, then $\mathcal{F}_c$ has a $((1-\delta)t_n,d_n)$ cutoff for any sequence satisfying $d_n=o(t_n)$ and $b_n=o(d_n)$.

\item[(2)]  If $\mathcal{F}_c$ has a $(r_n,b_n)$ cutoff, then $\mathcal{F}_\delta$ has a $(r_n/(1-\delta),d_n)$ cutoff for any sequence satisfying $d_n=o(r_n)$ and $b_n=o(d_n)$.
\end{itemize}
\end{prop}
\begin{proof}
Immediately from Theorem \ref{t-window} and Proposition \ref{p-main}.
\end{proof}

\begin{proof}[Proof of Theorem \ref{t-wcomp}]
We prove (1)$\Ra$(2), while the reasoning for (2)$\Ra$(1) is similar. Suppose that $\mathcal{F}_\delta$ has a $(t_n,b_n)$ cutoff with $t_n\ra\infty$. Fix $\epsilon\in(0,1)$ and set $c_n=|T_{n\text{\tiny TV}}^{(c)}(\epsilon)-(1-\delta)t_n|$. By \cite[Proposition 2.3]{CSal08}, it remains to show that $c_n=O(b_n)$. Assume the inverse, that is, there is a subsequence $\xi=\{n_l|l=1,2,...\}$ such that $c_{n_l}/b_{n_l}\ra\infty$ as $l\ra\infty$. Let $\mathcal{F}_\delta(\xi),\mathcal{F}_c(\xi)$ be families of $\mathcal{F}_\delta,\mathcal{F}_c$ restricted to $\xi$. This implies $\mathcal{F}_\delta(\xi)$ has a $(t_{n_l},b_{n_l})$ cutoff, but $\mathcal{F}_c(\xi)$ has no $\left((1-\delta)t_{n_l},\sqrt{b_{n_l}c_{n_l}}\right)$ cutoff, a contradiction with Proposition \ref{p-wcomp}.
\end{proof}

\subsection{Chains with specified initial states}\label{ss-initial}

For any probability $\mu$ on a countable set $\Omega$, we write $(\mu,\Omega,K,\pi)$ as an irreducible Markov chain on $\Omega$ with transition matrix $K$, stationary distribution $\pi$ and initial distribution $\mu$. The total variation distances for the associated $\delta$-lazy walk and continuous time chain are defined by
\begin{equation}\label{eq-initial}
    d_{\text{\tiny TV}}^{(\delta)}(\mu,n)=\|\mu K_\delta^n-\pi\|_{\text{\tiny TV}},\quad
    d_{\text{\tiny TV}}^{(c)}(\mu,t)=\|\mu H_t-\pi\|_{\text{\tiny TV}}.
\end{equation}
Denoted by $T_{\text{\tiny TV}}^{(\delta)}(\mu,\epsilon),T_{\text{\tiny TV}}^{(c)}(\mu,\epsilon)$ are the corresponding mixing times and the concept of cutoffs can be defined similarly as Definition \ref{d-cutoff} according to (\ref{eq-initial}). It is an easy exercise to achieve a similar version of Lemma \ref{l-cut1} for cutoffs with specified initial distributions. The proofs for Propositions \ref{p-main}-\ref{p-wcomp} and Corollary \ref{c-main} can be adapted to the case when the initial distribution is prescribed. This gives the following theorems.

\begin{thm}\label{t-ini}
Let $\mathcal{F}=\{(\mu_n,\Omega_n,K_n,\pi_n)|n=1,2,...\}$ be a family of irreducible Markov chains and  $\mathcal{F}_\delta,\mathcal{F}_c$ be families of associated $\delta$-lazy walks and continuous time chains.
\begin{itemize}
\item[(1)] $\mathcal{F}_\delta$ has a cutoff (resp. precutoff) iff $\mathcal{F}_c$ has a cutoff (resp. precutoff).

\item[(2)] If $\mathcal{F}_\delta$ has a cutoff, then $T_{n,\textnormal{\tiny TV}}^{(c)}(\mu_n,\epsilon)\sim(1-\delta)T_{n,\textnormal{\tiny TV}}^{(\delta)}(\mu_n,\epsilon)$ as $n$ tends to $\infty$ for all $\epsilon\in(0,1)$.
\end{itemize}
Let $t_n\ra\infty$ and $b_n>0$.
\begin{itemize}
\item[(3)] $\mathcal{F}_\delta$ has a $(t_n,b_n)$ cutoff iff $\mathcal{F}_c$ has a $((1-\delta)t_n,b_n)$ cutoff.

\item[(4)] If $\mathcal{F}_\delta$ has a $(t_n,b_n)$ cutoff, then $\sqrt{t_n}=O(b_n)$.
\end{itemize}
\end{thm}

\subsection{Proofs}

This subsection collects required techniques for the proof of theorems in Sections \ref{ss-cutofftime}-\ref{ss-window}.

\begin{lem}\label{l-meas}
Let $a_{n,m}\in[0,1]$, $t_n>0$ and $c\in\mathbb{R}$. Suppose that $t_n\ra\infty$. Then, there is a subsequence $n_k$ of positive integers and a nonnegative measurable function $f$ bounded by $1$ such that
\[
    \lim_{k\ra\infty}\sum_{m=0}^\infty\left(e^{-t_{n_k}-c\sqrt{t_{n_k}}}\frac{(t_{n_k}+c\sqrt{t_{n_k}})^m}
    {m!}\right)a_{n_k,m}=\frac{1}{\sqrt{2\pi}}\int_{-\infty}^\infty e^{-(x-c)^2/2}f(x)dx,
\]
for all $c\in\mathbb{R}$, and
\begin{align}
    &\lim_{k\ra\infty}\sum_{m\ge 0}\binom{[(t_{n_k}+c\sqrt{t_{n_k}})/(1-\delta)]}{m}(1-\delta)^m
    \delta^{[(t_{n_k}+c\sqrt{t_{n_k}})/(1-\delta)]-m}a_{n_k,m}\notag\\
    =&\frac{1}{\sqrt{2\pi\delta}}\int_{-\infty}^\infty e^{-(x-c)^2/(2\delta)}f(x)dx,\notag
\end{align}
for all $c\in\mathbb{R},\,\delta\in(0,1)$, where $[z]$ is any of $\lceil z\rceil,\lfloor z\rfloor$.
\end{lem}
\begin{proof}
For $n\ge 1$ and any Borel set $A\subset\mathbb{R}$, set
\[
    \mu_n(A)=\frac{1}{\sqrt{t_n}}\sum_{m:m-t_n/\sqrt{t_n}\in A}a_{n,m}.
\]
Let $n_k$ be a subsequence of $\mathbb{N}$ such that
\begin{equation}\label{eq-meas}
    \lim_{k\ra\infty}\mu_{n_k}((a,b])=\mu((a,b]),\quad\forall a,b\in\mathbb{Q},\,a<b.
\end{equation}
Clearly, $\mu((a,b])\le b-a$ for $a<b$ and $a,b\in\mathbb{Q}$. This implies the convergence in (\ref{eq-meas}) holds for all $a<b$ and $\mu((a,b])\le b-a$. As a consequence of the Carath\'{e}odory extension theorem, $\mu$ can be extended to a measure on $\mathbb{R}$. It is obvious that $\mu$ is absolutely continuous with respect to the Lebesgue measure and we write $f$ as the Radon-Nykodym derivative.

Let $\epsilon>0$ and choose $M>0$ such that, for $n\ge M$,
\[
    \sum_{m:|m-t_n|/\sqrt{t_n}\notin(-M,M]}e^{-t_n-c\sqrt{t_n}}\frac{(t_n+c\sqrt{t_n})^m}{m!}<\epsilon.
\]
For any integer $N>1$, set $x_i=iM/N$ and $A_{n,i}=\{m\ge 0||m-t_n|/\sqrt{t_n}\in(x_i,x_{i+1}]\}$. By Stirling's formula, it is easy to see that
\[
    e^{-t_n-c\sqrt{t_n}}\frac{(t_n+c\sqrt{t_n})^m}{m!}=\frac{1+o(1)}{\sqrt{2\pi t_n}}\exp\left\{-\frac{1}{2}\left(\frac{m-t_n}{\sqrt{t_n}}-c\right)^2\right\}\quad\text{as }n\ra\infty,
\]
where $o(1)$ is uniformly for $m\in A_{n,i}$ and $-N\le i<N$. This implies
\[
    \sum_{m\in A_{n,i}}\left(e^{-t_n-c\sqrt{t_n}}\frac{(t_n+c\sqrt{t_n})^m}{m!}\right)a_{n,m}\begin{cases}\le
    M_i\mu_n(A_{n,i})/\sqrt{2\pi}+o(1)\\\ge m_i\mu_n(A_{n,i})/\sqrt{2\pi}+o(1)\end{cases},
\]
where $U_i=\sup\{e^{-(x-c)^2/2}|x\in(x_i,x_{i+1}]$ and $L_i=\inf\{e^{-(x-c)^2/2}|x\in(x_i,x_{i+1}]\}$. Summing up $i$ and replacing $n$ with $n_k$ yields
\[
    \limsup_{k\ra\infty}\sum_{m=0}^\infty\left(e^{-t_{n_k}-c\sqrt{t_{n_k}}}\frac{(t_{n_k}+c\sqrt{t_{n_k}})^m}
    {m!}\right)a_{n_k,m}\le\frac{1}{\sqrt{2\pi}}\sum_{i=-N}^{N-1}M_i\mu((x_i,x_{i+1}])+\epsilon
\]
and
\[
    \liminf_{k\ra\infty}\sum_{m=0}^\infty\left(e^{-t_{n_k}-c\sqrt{t_{n_k}}}\frac{(t_{n_k}+c\sqrt{t_{n_k}})^m}
    {m!}\right)a_{n_k,m}\ge\frac{1}{\sqrt{2\pi}}\sum_{i=-N}^{N-1}m_i\mu((x_i,x_{i+1}])-\epsilon.
\]
Letting $N\ra\infty$ and then $\epsilon\ra 0$ gives the desired limit. The proof of the second limit is similar and omitted.
\end{proof}

\begin{lem}\label{l-conv}
Let $f$ be a bounded nonnegative measurable function and set $F(t)=\int_{-\infty}^\infty e^{-(x-t)^2}f(x)dx$. If $F$ is constant, then $f$ is constant almost everywhere.
\end{lem}
\begin{proof}
Set $A=F(t)$, $B^{-1}=\int_{-\infty}^\infty e^{-x^2/2}f(x)dx$ and write
\[
    e^{-(x-t/2)^2}f(x)=B^{-1}\sqrt{2\pi}e^{t^2/4}\left(\frac{1}{\sqrt{2\pi}}
    e^{-(t-x)^2/2}\right)\left(Be^{-x^2/2}f(x)\right).
\]
Note that $AB/(\sqrt{2\pi}e^{t^2/4})$ is the density of $X+Y$, where $X$ has the standard normal distribution, $Y$ is continuous with density function $Be^{-x^2/2}f(x)$ and $X,Y$ are independent. This implies $AB=1/\sqrt{2}$ and
\[
    e^{-u^2}=\mathbb{E}(e^{iu(X+Y)})=e^{-u^2/2}\mathbb{E}(e^{iuY}),\quad\forall u\in\mathbb{R}.
\]
Clearly, $Y$ has the standard normal distribution and, thus, $f$ is a constant a.e.
\end{proof}

\subsection{A remark on the spectral gap and mixing time}
In this subsection, we make a comparison of spectral gaps between continuous time chains and $\delta$-lazy discrete time chains. Let $(\Omega,K,\pi)$ be an irreducible and reversible finite Markov chain with spectral gap $\lambda$, the smallest non-zero eigenvalue of $I-K$. First, we consider the continuous time case. Since $(K,\pi)$ is reversible, there is a function $f$ defined on $\{0,1,...,n\}$ such that $Kf=(1-\lambda)f$. This implies
\[
 d_{\text{\tiny TV}}^{(c)}(t)=\frac{1}{2}\|H_t-\pi\|_{\infty\ra\infty}\ge\frac{\|(H_t-\pi)f\|_\infty}{2\|f\|_\infty}=\frac{e^{-\lambda t}}{2},
\]
where $\|A\|_{\infty\ra\infty}:=\sup\{\|Ag\|_{\infty}:\|g\|_\infty=1\}$. Consequently, we obtain
\[
 T_{\text{\tiny TV}}^{(c)}(\epsilon)\ge\frac{-\log(2\epsilon)}{\lambda}.
\]
For the lazy discrete time case, a similar discussion yields
\[
 d_{\text{\tiny TV}}^{(\delta)}(t)\ge \beta_\delta^t/2,\quad T_{\text{\tiny TV}}^{(\delta)}(\epsilon)\ge\left\lfloor\frac{\log(2\epsilon)}{\log \beta_\delta}\right\rfloor,
\]
where $\beta_\delta$ is the second largest absolute value of all nontrivial eigenvalue values of $K_\delta$. By setting $\delta_0=\inf\{\delta\in(0,1)|\beta_\delta=1-(1-\delta)\lambda\}$, it is easy to see that $\delta_0\le 1/2$ and, for $\delta\in[\delta_0,1)$, $\beta_\delta=1-(1-\delta)\lambda$. As a function of $\delta$, $\beta_\delta$ is decreasing on $(0,\delta_0)$ and increasing on $(\delta_0,1)$. Note that $|1-(1-\delta)\lambda|\le \beta_\delta\le\max\{1-2\delta,1-(1-\delta)\lambda\}$. The first inequality implies $1-\beta_\delta\le(1-\delta)\lambda$. Using the second inequality, if $\beta_{\delta}>1-2\delta$, then $1-\beta_\delta=(1-\delta)\lambda$. If $\beta_\delta\le 1-2\delta$, then $1-\beta_\delta\ge 2\delta\ge \delta\lambda$, where the last inequality uses the fact $\lambda\le 2$. We summarize the discussion in the following lemma.
\begin{lem}\label{l-gapsing}
Let $K$ be an irreducible transition matrix on a finite set $\Omega$ with stationary distribution $\pi$. For $\delta\in(0,1)$, let $K_\delta$ be the $\delta$-lazy walk given by \textnormal{(\ref{eq-lazyK})}. Suppose $(\pi,K)$ is reversible, that is, $\pi(x)K(x,y)=\pi(y)K(y,x)$ for all $x,y\in\Omega$ and let $\lambda$ be the smallest non-zero eigenvalue of $I-K$ and $\beta_\delta$ be the largest absolute value of all nontrivial eigenvalues of $K_\delta$. Then, it holds true that
\[
 \min\left\{1-\delta,\delta\right\}\lambda\le 1-\beta_\delta\le 1-|1-(1-\delta)\lambda|\le(1-\delta)\lambda,\quad\forall\delta\in(0,1).
\]
Furthermore, for $\epsilon\in(0,1/2)$,
\[
 T_{\textnormal{\tiny TV}}^{(c)}(\epsilon)\ge\frac{-\log(2\epsilon)}{\lambda},\quad T_{\textnormal{\tiny TV}}^{(\delta)}(\epsilon)\ge\left\lfloor\frac{\log(2\epsilon)}{\log \beta_\delta}\right\rfloor\ge\left\lfloor\frac{-\log(2\epsilon)}{2\max\{1-\delta,\log(2/\delta)\}\lambda}\right\rfloor,
\]
where the last inequality assumes $|\Omega|\ge 2/\delta$.
\end{lem}

\begin{proof}
It remains to prove the second inequality in the lower bound of the mixing time for the $\delta$-lazy chain. Note that if $\lambda\le 1/2$, then
\[
 -\log\beta_\delta\le -\log(1-(1-\delta)\lambda)\le 2(1-\delta)\lambda,
\]
where the last inequality uses the fact $\log(1-x)\ge -2x$ for $x\in(0,1/2)$. For $\lambda\ge 1/2$, let $\theta_1(\delta),...,\theta_{|\Omega|}(\delta)$ be eigenvalues of $K_\delta$. Then, $\theta_i(\delta)=\delta+(1-\delta)\theta_i(0)$ and
$\sum_{i=1}^{|\Omega|}\theta_i(0)\ge 0$. See \cite{SC97} for a reference on the second inequality. This implies
\[
 1+(|\Omega|-1)\beta_\delta\ge \sum_{i=1}^{|\Omega|}\theta_i(\delta)\ge|\Omega|\delta.
\]
Assuming $|\Omega|\ge 2/\delta$, the above inequality yields
\[
 \beta_\delta\ge\frac{|\Omega|\delta-1}{|\Omega|-1}\ge\frac{\delta}{2},\quad -\log\beta_\delta\le\left(2\log\frac{2}{\delta}\right)\lambda.
\]
\end{proof}

\section{Finite birth and death chains}

In this section, we consider the total variation cutoff for birth and death chains. A birth and death chain on $\{0,1,...,n\}$ with birth rate $p_i$, death rate $q_i$ and holding rate $r_i$ is a Markov chain with transition matrix $K$ given by
\begin{equation}\label{eq-bdc}
 K(i,i+1)=p_i,\quad K(i,i-1)=q_i,\quad K(i,i)=r_i,\quad\forall 0\le i\le n,
\end{equation}
where $p_i+q_i+r_i=1$ and $p_n=q_0=0$. It is obvious that $K$ is irreducible if and only if $p_iq_{i+1}>0$ for $0\le i<n$. Under the assumption of irreducibility, the unique stationary distribution $\pi$ of $K$ is given by $\pi(i)=c(p_0\cdots p_{i-1})/(q_1\cdots q_i)$, where $c$ is a positive constant such that $\sum_{i=0}^n\pi(i)=1$.

In the next two subsections, we recall some results developed in \cite{DS06,DLP10} and make an improvement on them using the result in Section 3. In the third subsection, we go back to the issue of cutoffs and make a comparison of total variation and separation cutoffs.

\subsection{The total variation cutoff}
Throughout this subsection, we let
\begin{equation}\label{eq-bdcfamily}
 \mathcal{F}=\{(\Omega_n,K_n,\pi_n)|n=1,2,...\}
\end{equation}
denote a family of irreducible birth and death chains with $\Omega_n=\{0,1,...,n\}$ and transition matrix
\begin{equation}\label{eq-nbdc}
 K_n(i,i+1)=p_{n,i},\quad K_n(i,i-1)=q_{n,i},\quad K_n(i,i)=r_{n,i},\quad\forall 0\le i\le n,
\end{equation}
where $p_{n,i}+q_{n,i}+r_{n,i}=1$ and $p_{n,n}=q_{n,0}=0$. Write $\lambda_n=\lambda(K_n)$ as the spectral gap of $K_n$. As before, $\mathcal{F}_c$ denotes the family of associated continuous semigroups and, for $\delta\in(0,1)$, $\mathcal{F}_\delta$ denotes the family of $\delta$-lazy chains. Recall one of the main results in \cite{DLP10} as follows.
\begin{thm}[Theorems 3-3.1 in \cite{DLP10}]\label{t-bdc-cut}
Consider the family in \textnormal{(\ref{eq-bdcfamily})}. For $n\ge 1$, let $\lambda_n$ be the smallest nonzero eigenvalue of $I-K_n$ and let $\beta_{n,\delta}$ be the second largest absolute value of all nontrivial eigenvalues of $K_{n,\delta}$. Then, $\mathcal{F}_c$ (resp. $\mathcal{F}_\delta$ with $\delta\in(0,1)$) has a total variation cutoff if and only if $T_{n,\textnormal{\tiny TV}}^{(c)}(1/4)\lambda_n\ra\infty$ (resp. $T_{n,\textnormal{\tiny TV}}^{(\delta)}(1/4)(1-\beta_{n,\delta})\ra\infty$). Moreover, if $\mathcal{F}_c$ (resp. $\mathcal{F}_\delta$) has a cutoff, then the window has size at most $\sqrt{T_{n,\textnormal{\tiny TV}}^{(c)}(1/4)/\lambda_n}$ (resp. $\sqrt{T_{n,\textnormal{\tiny TV}}^{(\delta)}(1/4)/(1-\beta_{n,\delta})}$).
\end{thm}
\begin{rem}
By Lemma \ref{l-gapsing}, the total variation cutoff in discrete time case is equivalent to $T_{n,\text{\tiny TV}}^{(\delta)}(1/4)\lambda_n\ra\infty$. By Theorems \ref{t-cutcomp}-\ref{t-bdc-cut} and Lemma \ref{l-gapsing}, if $\mathcal{F}_c$ or $\mathcal{F}_\delta$ has a cutoff, then the window size is at most $\sqrt{T_{n,\text{\tiny TV}}^{(c)}(1/4)/\lambda_n}$ or $\sqrt{T_{n,\text{\tiny TV}}^{(\delta)}(1/4)/\lambda_n}$.
\end{rem}

\begin{rem}
There are examples with cutoffs, but the order of the optimal window size is less than $\sqrt{T_{n,\text{\tiny TV}}^{(c)}(1/4)\lambda_n}$. See Remark \ref{r-window}.
\end{rem}

The combination of the above theorem and Theorem \ref{t-cutcomp} yields
\begin{thm}\label{t-bdc-cut2}
Referring to Theorem \ref{t-bdc-cut}, the following are equivalent.
\begin{itemize}
\item[(1)] $\mathcal{F}_c$ has a total variation cutoff.

\item[(2)] $\mathcal{F}_\delta$ has a total variation cutoff.

\item[(3)] $\mathcal{F}_c$ has a total variation precutoff.

\item[(4)] $\mathcal{F}_\delta$ has a total variation precutoff.

\item[(5)] $T_{n,\textnormal{\tiny TV}}^{(c)}(\epsilon)\lambda_n\ra\infty$ for some $\epsilon\in(0,1)$.

\item[(6)] $T_{n,\textnormal{\tiny TV}}^{(\delta)}(\epsilon)\lambda_n\ra\infty$ for some $\epsilon\in(0,1)$.
\end{itemize}
\end{thm}
\begin{proof}[Proof of Theorem \ref{t-bdc-cut2}]
It remains to show (3)$\Ra$(5) and this is given by the inequality $d_{n,\text{\tiny TV}}^{(c)}(t)\ge e^{-\lambda_nt}/2$.
\end{proof}

\begin{thm}\label{t-bdc-cut3}
Consider the family in \textnormal{(\ref{eq-bdcfamily})}. It holds true that $T_{n,\textnormal{\tiny TV}}^{(c)}(\epsilon/2)\asymp T_{n,\textnormal{\tiny TV}}^{(\delta)}(\eta/2)$ for all $\epsilon,\eta,\delta\in(0,1)$. Furthermore, if there is $\epsilon_0\in(0,1)$ such that $T_{n,\textnormal{\tiny TV}}^{(c)}(\epsilon_0/2)\lambda_n$ or $T_{n,\textnormal{\tiny TV}}^{(\delta)}(\epsilon_0/2)\lambda_n$ is bounded, then $T_{n,\textnormal{\tiny TV}}^{(c)}(\epsilon/2)\asymp 1/\lambda_n$ and $T_{n,\textnormal{\tiny TV}}^{(\delta)}(\epsilon/2)\asymp 1/\lambda_n$ for all $\epsilon,\delta\in(0,1)$.
\end{thm}
\begin{proof}[Proof of Theorem \ref{t-bdc-cut3}]
Assume that there is a subsequence $n_k$ and $\epsilon,\eta\in(0,1/2)$ such that either $T_{n_k,\text{\tiny TV}}^{(c)}(\epsilon)/T_{n_k,\text{\tiny TV}}^{(\delta)}(\eta)\ra\infty$ or $T_{n_k,\text{\tiny TV}}^{(\delta)}(\eta)/T_{n_k,\text{\tiny TV}}^{(c)}(\epsilon)\ra\infty$. By Lemma \ref{l-gapsing}, we have  $T_{n_k,\text{\tiny TV}}^{(c)}(\epsilon)\lambda_{n_k}\ra\infty$ or $T_{n_k,\text{\tiny TV}}^{(\delta)}(\eta)\lambda_{n_k}\ra\infty$. In either case, Theorems \ref{t-cutcomp}-\ref{t-bdc-cut} imply that the subfamily indexed by $(n_k)_{k=1}^\infty$ has a cutoff in both continuous time and $\delta$-lazy discrete time cases. As a consequence of Theorem \ref{t-cutcomp}, we obtain $T_{n_k,\text{\tiny TV}}^{(c)}(\epsilon)\sim(1-\delta)T_{n_k,\text{\tiny TV}}^{(\delta)}(\eta)$, which contradicts with the assumption.
\end{proof}

Concerning the window size, a combination of Theorem \ref{t-window} and Theorem \ref{t-bdc-cut} yields
\begin{thm}\label{t-bdc-win}
Let $\mathcal{F},\lambda_n$ be as in Theorem \ref{t-bdc-cut}. Suppose that $\mathcal{F}_c$ or $\mathcal{F}_\delta$ has a total variation cutoff and $\lambda_n\asymp 1$. Then, for any $\epsilon,\eta\in(0,1)$ with $\epsilon\ne\eta$,
\[
 \left|T_{n,\textnormal{\tiny TV}}^{(c)}(\epsilon)-T_{n,\textnormal{\tiny TV}}^{(c)}(\eta)\right|\asymp \sqrt{T_{n,\textnormal{\tiny TV}}^{(c)}(\epsilon)}\asymp \left|T_{n,\textnormal{\tiny TV}}^{(\delta)}(\epsilon)-T_{n,\textnormal{\tiny TV}}^{(\delta)}(\eta)\right|.
\]
\end{thm}

\subsection{The separation cutoff}
In this subsection, we apply the results obtained in the previous subsection to the separation cutoff. First, we give a definition of the separation in the following. Given an irreducible finite Markov chain $K$ on $\Omega$ with initial distribution $\mu$ and stationary distribution $\pi$, the separation distance at time $m$ is defined by
\[
 d_{\text{sep}}(\mu,m):=\max_{x\in\Omega}\left\{1-\frac{\mu K^m(x)}{\pi(x)}\right\}.
\]
Aldous and Diaconis \cite{AD87} introduce the concept of the strong stationary time to identify the separation distance. Set $d_{\text{sep}}(m)=\max_id_{\text{sep}}(i,m)$. A well-known bound on the separation is achieved by Aldous and Fill in Lemma 7 of \cite[Chapter 4]{AF}, which says
\begin{equation}\label{eq-sep}
 \bar{d}(m)\le d_{\text{sep}}(m),\quad d_{\text{sep}}(2m)\le 1-(1-\bar{d}(m))^2,
\end{equation}
where $\bar{d}(m):=\max_{i,j}\|K^m(i,\cdot)-K^m(j,\cdot)\|_{\text{\tiny TV}}$. It is clear from the definitions that $d_{\text{\tiny TV}}(m)\le \bar{d}(m)\le 2d_{\text{\tiny TV}}(m)$. Let $T_{\text{sep}}(\epsilon)$ be the separation mixing time. The above inequalities imply
\begin{equation}\label{eq-septv}
 T_{\text{\tiny TV}}(\epsilon)\le T_{\text{sep}}(\epsilon)\le 2T_{\text{\tiny TV}}(\epsilon/4),\quad\forall \epsilon\in(0,1).
\end{equation}
Note that the above discussions are also valid for the continuous time case. As the separation distance is between $(0,1)$, the separation cutoff is similar to the total variation cutoff as in Definition \ref{d-cutoff}. By (\ref{eq-septv}), we obtain the following lemma.
\begin{lem}\label{l-precut}
Let $\mathcal{F}$ be a family of finite Markov chains in either discrete or continuous time case. Assume that $T_{n,\textnormal{\tiny TV}}(\epsilon)\ra\infty$ or $T_{n,\textnormal{sep}}(\epsilon)\ra\infty$ for some $\epsilon\in(0,1)$ in discrete time case. Then, $\mathcal{F}$ has a total variation precutoff if and only if $\mathcal{F}$ has a separation precutoff.
\end{lem}

For birth and death chains, the application of (\ref{eq-septv}) to Theorem \ref{t-bdc-cut3} leads to the following theorem.

\begin{thm}\label{t-bdc-sep}
Theorem \ref{t-bdc-cut3} also holds in separation. Furthermore, for $\epsilon,\eta\in(0,1/2)$,
$T_{n,\textnormal{\tiny TV}}^{(c)}(\epsilon)\asymp T_{n,\textnormal{sep}}^{(c)}(\eta)$.
\end{thm}

Let $K$ be an irreducible birth and death chain on $\{0,1,...,n\}$ with stationary distribution $\pi$. The authors in \cite{DLP10} obtain the following fact
\begin{equation}\label{eq-maxsep}
d_{\text{sep}}^{(c)}(t)=1-\frac{H_t(0,n)}{\pi(n)},\quad d_{\text{sep}}^{(\delta)}(m)=1-\frac{K_\delta^m(0,n)}{\pi(n)}\quad\forall\delta\in[1/2,1).
\end{equation}
The authors in \cite{DS06} provide a criterion on the separation cutoff for continuous time chains and monotone discrete time chains. The result says that a separation cutoff exists if and only if the product of the spectral gap and the separation mixing time tends to infinity. The next theorem is a consequence of this fact and Theorems \ref{t-bdc-cut2} and \ref{t-bdc-sep}, which is also obtained in \cite{DLP10}.

\begin{thm}\label{t-bdc-sep2}
Let $\mathcal{F}$ be a family of birth and death chains given by \textnormal{(\ref{eq-bdcfamily})}. The following are equivalent.
\begin{itemize}
\item[(1)] $\mathcal{F}_c$ has a cutoff in total variation.

\item[(2)] For $\delta\in(0,1)$, $\mathcal{F}_\delta$ has a cutoff in total variation.

\item[(3)] $\mathcal{F}_c$ has a cutoff in separation.

\item[(4)] For $\delta\in[1/2,1)$, $\mathcal{F}_\delta$ has a cutoff in separation.
\end{itemize}
\end{thm}

The next theorem is a simple corollary of Theorems \ref{t-bdc-cut2}-\ref{t-bdc-sep2} and Lemma \ref{l-precut}.

\begin{thm}
Theorem \ref{t-bdc-cut2} also holds in separation distance with $\delta\in[1/2,1)$.
\end{thm}

\subsection{The cutoff time in total variation and separation}

In this subsection, we introduce a spectral representation of the total variation mixing time. Let $K$ be the transition kernel of an irreducible birth and death chain on $\{0,1,...,n\}$. Suppose that $K$ is irreducible with stationary distribution $\pi$ and let $0<\lambda_1<\cdots<\lambda_n$ be the eigenvalues of $I-K$. Consider the continuous time case. Using \cite[Theorem 4.1]{DS06} and \cite[Corollary 4.5]{DLP10}, we have
\[
 d_{\text{sep}}^{(c)}(t)=1-\frac{H_t(0,n)}{\pi(n)}=1-\frac{H_t(n,0)}{\pi(0)}=\mathbb{P}(S>t),
\]
where $S$ is a sum of $n$ independent exponential random variables with parameters $\lambda_1,...,\lambda_n$.
By the one-sided Chebyshev inequality, one has
\[
 \mathbb{E}S-\sqrt{\text{Var}(S)/(1/\epsilon-1)}\le T_{\text{sep}}^{(c)}(\epsilon)\le\mathbb{E}S+\sqrt{(1/\epsilon-1)\text{Var}(S)},\quad\forall \epsilon\in(0,1).
\]
Note that
\[
 \mathbb{E}S=\sum_{i=1}^n\frac{1}{\lambda_i},\quad\text{Var}(S)=\sum_{i=1}^n\frac{1}{\lambda_i^2}\le(\mathbb{E}S)^2.
\]
Clearly, this implies
\begin{equation}\label{eq-sepES}
 \frac{(\sqrt{1-\epsilon}-\sqrt{\epsilon})\mathbb{E}S}{\sqrt{1-\epsilon}}\le T_{\text{sep}}^{(c)}(\epsilon)\le\frac{(\sqrt{\epsilon}+\sqrt{1-\epsilon})\mathbb{E}S}{\sqrt{\epsilon}},\quad\forall\epsilon\in(0,1).
\end{equation}
The above equation says that, given $\epsilon\in(0,1/2)$, the separation mixing time is bounded by $\sum_{i=1}^n\lambda_i^{-1}$ up to universal constants. The above discussion is also valid for discrete time case with the assumption that $K(i,i+1)+K(i+1,i)\le 1$ for $0\le i<n$. See \cite{DS06} for the details. The next proposition is an application of (\ref{eq-septv}) and (\ref{eq-sepES}).

\begin{prop}\label{p-tvsep}
Let $K$ be an irreducible birth and death chain on $\{0,1,...,n\}$. Let $0,\lambda_1,...,\lambda_n$ be eigenvalues of $K$ and set $s=\sum_{i=1}^n\lambda_i^{-1}$. Then,
\[
 \left(\frac{\sqrt{1-\epsilon}-\sqrt{\epsilon}}{\sqrt{1-\epsilon}}\right)s\le T_{\textnormal{sep}}^{(c)}(\epsilon)\le\left(\frac{\sqrt{\epsilon}+\sqrt{1-\epsilon}}{\sqrt{\epsilon}}\right)s,\quad \forall \epsilon\in(0,1/2),
\]
and
\[
 \frac{1}{2}\left(\frac{\sqrt{1-4\epsilon}-\sqrt{4\epsilon}}{\sqrt{1-4\epsilon}}\right)s\le T_{\textnormal{\tiny TV}}^{(c)}(\epsilon)\le\left(\frac{\sqrt{\epsilon}+\sqrt{1-\epsilon}}{\sqrt{\epsilon}}\right)s,\quad\forall \epsilon\in(0,1/8).
\]
The above also holds in discrete time case with the assumption that $K(i,i+1)+K(i+1,i)\le 1$ for $0\le i<n$.
\end{prop}

Applying Proposition \ref{p-tvsep} to Theorems \ref{t-bdc-cut2}-\ref{t-bdc-sep2} yields the following theorem, where the result in separation is included in \cite{DS06} and the result in total variation is implicitly obtained in \cite{DLP10}.

\begin{thm}[Cutoffs from the spectrum]\label{t-mixingtvsep3}
Let $\mathcal{F}$ be the family in \textnormal{(\ref{eq-bdcfamily})}. For $n\ge 1$, let $\lambda_{n,1},...,\lambda_{n,n}$ be non-zero eigenvalues of $I-K_n$ and set
\[
 \lambda_n=\min_{1\le i\le n}\lambda_{n,i},\quad s_n=\frac{1}{\lambda_{n,1}}+\cdots+\frac{1}{\lambda_{n,n}}.
\]
Then, the following are equivalent.
\begin{itemize}
\item[(1)] $\mathcal{F}_c$ has a total variation cutoff.

\item[(2)] For $\delta\in(0,1)$, $\mathcal{F}_\delta$ has a total variation cutoff.

\item[(3)] $\mathcal{F}_c$ has a total variation precutoff.

\item[(4)] For $\delta\in(0,1)$, $\mathcal{F}_\delta$ has a total variation precutoff.

\item[(5)] $s_n\lambda_n\ra\infty$.
\end{itemize}

The above also holds in separation with $\delta\in[1/2,1)$. In particular, if \textnormal{(5)} holds, then, for $\epsilon\in(0,1)$,
\[
 \frac{1}{2}\le\liminf_{n\ra\infty}\frac{T_{n,\textnormal{\tiny TV}}^{(c)}(\epsilon)}{s_n}\le\limsup_{n\ra\infty}\frac{T_{n,\textnormal{\tiny TV}}^{(c)}(\epsilon)}{s_n}\le 1.
\]
\end{thm}

The last result establishes a relation between the mixing time and birth and death rates. Consider an irreducible birth and death chain $(X_m)_{m=0}^\infty$ on $\{0,1,...,n\}$ with transition matrix $K$ and stationary distribution $\pi$. Let $N_t$ be a Poisson process of parameter 1 that is independent of $X_m$ and set, for $0\le i\le n$,
\[
 \tau_i:=\inf\{t\ge 0|X_{N_t}=i\}.
\]
Brown and Shao discuss the distribution of $\tau_i$ in \cite{BS87} and obtain the following result.
\[
 \mathbb{P}_0(\tau_n>t)=\sum_{j=1}^n\left(\prod_{k\ne j}\frac{\theta_k}{\theta_k-\theta_j}\right)e^{-\theta_jt},
\]
where $\mathbb{P}_i$ is the conditional probability given $X_0=i$ and $\theta_1,...,\theta_n$ are eigenvalues of the submatrix of $I-K$ indexed by $\{0,1,...,n-1\}$. Let $\mathbb{E}_i$ be the conditional expectation given $X_0=i$. Clearly, this implies $\mathbb{E}_0\tau_n=\sum_{j=1}^n1/\theta_j$. Note that $\mathbb{E}_0\tau_n$ can be formulated by the birth and death rates using the strong Markov property. This leads to
\begin{equation}\label{eq-passage}
 \mathbb{E}_0\tau_n=\sum_{j=1}^n\frac{1}{\theta_j}=\sum_{k=0}^{n-1}\frac{\pi([0,k])}{\pi(k)p_k},
\end{equation}
where $\pi(A):=\sum_{i\in A}\pi(i)$.

Fix $0\le i_0\le n$. By (\ref{eq-passage}), we have
\[
 \mathbb{E}_0\tau_{i_0}=\sum_{i=1}^{i_0}\frac{1}{\lambda_i'},\quad\mathbb{E}_n\tau_{i_0}=\sum_{i=1}^{n-i_0}\frac{1}{\lambda_i''},
\]
where $\lambda_1',...,\lambda_{i_0}'$ and $\lambda_1'',...,\lambda_{n-i_0}''$ are eigenvalues of the submatrices of $I-K$ indexed respectively by $\{0,...,i_0-1\}$ and $\{i_0+1,...,n\}$. Let $\bar{\lambda}_1\le\cdots\le\bar{\lambda}_n$ be a rearrangement of $\lambda'_1,...,\lambda_{i_0}',\lambda_1'',...,\lambda_{n-i_0}''$. Clearly, $\bar{\lambda}_1,...,\bar{\lambda}_n$ are eigenvalues of the submatrix obtained by removing the $i_0$-th row and the $i_0$-th column of $I-K$. Let $\lambda_1<\cdots<\lambda_n$ be nonzero eigenvalues of $I-K$. By Theorem 4.3.8 in \cite{HJ90}, we have $\bar{\lambda}_i\le\lambda_i\le\bar{\lambda}_{i+1}$ and this leads to
\[
 \sum_{i=2}^n\frac{1}{\bar{\lambda}_i}\le\sum_{i=1}^n\frac{1}{\lambda_i}\le\sum_{i=1}^n\frac{1}{\bar{\lambda}_i}=\sum_{k=0}^{i_0-1}\frac{\pi([0,k])}{\pi(k)p_k}+\sum_{k=i_0+1}^n\frac{\pi([k,n])}{\pi(k)q_k},
\]
where the first equality uses (\ref{eq-passage}). By Proposition \ref{p-tvsep}, we obtain, for $\epsilon\in(0,1)$,
\[
 T_{\text{\tiny TV}}^{(c)}(\epsilon)\le T_{\text{sep}}^{(c)}(\epsilon)\le\left(\frac{\sqrt{\epsilon}+\sqrt{1-\epsilon}}{\sqrt{\epsilon}}\right)\min_{0\le i\le n}\left\{\sum_{k=0}^{i-1}\frac{\pi([0,k])}{\pi(k)p_k}+\sum_{k=i+1}^n\frac{\pi([k,n])}{\pi(k)q_k}\right\}.
\]
The above discussion also holds in discrete time case with the assumption that $p_i+q_{i+1}\le 1$ for all $0\le i<n$. This includes the $\delta$-lazy chain for $\delta\in[1/2,1)$ and we apply it to get the following corollary.

\begin{cor}\label{c-tn}
Let $\mathcal{F}=\{(\Omega_n,K_n,\pi_n)|n=1,2,...\}$ be a family of irreducible birth and death chain in \textnormal{(\ref{eq-bdcfamily})} with birth, death and holding rates $p_{n,i},q_{n,i},r_{n,i}$. For $n\ge 1$, set
\[
 t_n=\min_{0\le i\le n}\left\{\sum_{k=0}^{i-1}\frac{\pi_n([0,k])}{\pi_n(k)p_{n,k}}+\sum_{k=i+1}^n\frac{\pi_n([k,n])}{\pi_n(k)q_{n,k}}\right\}.
\]
If $\mathcal{F}_c$ or $\mathcal{F}_\delta$ has a total variation cutoff, then, for $\epsilon\in(0,1)$ and $\delta\in[1/2,1)$,
\[
 \limsup_{n\ra\infty}\frac{T_{n,\textnormal{sep}}^{(c)}(\epsilon)}{t_n}\le 1,\quad
\limsup_{n\ra\infty}\frac{T_{n,\textnormal{sep}}^{(\delta)}(\epsilon)}{t_n}\le\frac{1}{1-\delta},
\]
and, for $\epsilon\in(0,1)$,
\[
 \limsup_{n\ra\infty}\frac{T_{n,\textnormal{\tiny TV}}^{(c)}(\epsilon)}{t_n}\le 1\quad
\limsup_{n\ra\infty}\frac{T_{n,\textnormal{\tiny TV}}^{(\delta)}(\epsilon)}{t_n}\le\frac{1}{1-\delta}.
\]
\end{cor}

\begin{rem}
In \cite{CSal12-3}, the constant $t_n$ in Corollary \ref{c-tn} is proved to be of the same order as the constant $s_n$ in Theorem \ref{t-mixingtvsep3} and the following term
\[
 \sum_{k=0}^{i_n-1}\frac{\pi_n([0,k])}{\pi_n(k)p_{n,k}}+\sum_{k=i_n+1}^n\frac{\pi_n([k,n])}{\pi_n(k)q_{n,k}},
\]
where $i_n$ satisfies $\pi_n([0,i_n])\ge 1/2$ and $\pi_n([i_n,n])\ge 1/2$.
\end{rem}
\begin{rem}
The bound in Corollary \ref{c-tn} is also be obtained implicitly in \cite{DLP10} using a coupling argument.
\end{rem}

\section*{Acknowledgements}
We thank the anonymous referee for her/his very careful reading of the manuscript.

\bibliographystyle{plain}
\bibliography{reference}

\begin{thebibliography}{10}

\bibitem{AF}
D.~Aldous and J.~A. Fill.
\newblock Reversible markov chains and random walks on graphs.
\newblock Monograph at http://www.stat.berkeley.edu/users/aldous/RWG/book.html.

\bibitem{AD87}
David Aldous and Persi Diaconis.
\newblock Strong uniform times and finite random walks.
\newblock {\em Adv. in Appl. Math.}, 8(1):69--97, 1987.

\bibitem{BS87}
M.~Brown and Y.-S. Shao.
\newblock Identifying coefficients in the spectral representation for first
  passage time distributions.
\newblock {\em Probab. Engrg. Inform. Sci.}, 1:69--74, 1987.

\bibitem{GY-PhD}
Guan-Yu Chen.
\newblock {\em The cutoff phenomenon for finite Markov chains}.
\newblock PhD thesis, Cornell University, 2006.

\bibitem{CSal08}
Guan-Yu Chen and Laurent Saloff-Coste.
\newblock The cutoff phenomenon for ergodic markov processes.
\newblock {\em Electron. J. Probab.}, 13:26--78, 2008.

\bibitem{CSal12-3}
Guan-Yu Chen and Laurent Saloff-Coste.
\newblock On the mixing time and spectral gap for birth and death chains.
\newblock In preparation, 2012.

\bibitem{D88}
Persi Diaconis.
\newblock {\em Group representations in probability and statistics}.
\newblock Institute of Mathematical Statistics Lecture Notes---Monograph
  Series, 11. Institute of Mathematical Statistics, Hayward, CA, 1988.

\bibitem{D96cutoff}
Persi Diaconis.
\newblock The cutoff phenomenon in finite {M}arkov chains.
\newblock {\em Proc. Nat. Acad. Sci. U.S.A.}, 93(4):1659--1664, 1996.

\bibitem{DS06}
Persi Diaconis and Laurent Saloff-Coste.
\newblock Separation cut-offs for birth and death chains.
\newblock {\em Ann. Appl. Probab.}, 16(4):2098--2122, 2006.

\bibitem{DLP10}
Jian Ding, Eyal Lubetzky, and Yuval Peres.
\newblock Total variation cutoff in birth-and-death chains.
\newblock {\em Probab. Theory Related Fields}, 146(1-2):61--85, 2010.

\bibitem{HJ90}
Roger~A. Horn and Charles~R. Johnson.
\newblock {\em Matrix analysis}.
\newblock Cambridge University Press, Cambridge, 1990.
\newblock Corrected reprint of the 1985 original.

\bibitem{SC97}
L.~Saloff-Coste.
\newblock Lectures on finite {M}arkov chains.
\newblock In {\em Lectures on probability theory and statistics (Saint-Flour,
  1996)}, volume 1665 of {\em Lecture Notes in Math.}, pages 301--413.
  Springer, Berlin, 1997.

\end{thebibliography}

\end{document}